\documentclass[12pt, oneside]{amsart}

\usepackage{ucs}
\usepackage{amssymb}
\usepackage{amsthm}
\usepackage{amsmath}
\usepackage{latexsym}
\usepackage[cp1251]{inputenc}
\usepackage{graphicx}
\usepackage{caption}
\usepackage{indentfirst}
\usepackage[left=2cm,right=2.5cm,top=2.5cm,bottom=2.5cm,bindingoffset=0cm]{geometry}
\usepackage{enumerate}
\usepackage{array,longtable,multirow}

\newtheorem{lemma}{Lemma}
\newtheorem{theo}{Theorem}
\newtheorem{prop}{Proposition}
\newtheorem{cor}{Corollary}

\begin{document}

\begin{center} {\large A generalization of the Arad--Ward theorem on Hall subgroups.}\\

\vspace{5mm}

N. Yang,

Jiangnan University, Wuxi, China

A.A. Buturlakin,\footnote{buturlakin@gmail.com, corresponding author}

Sobolev Institute of Mathematics, Novosibirsk, Russia

\end{center}


\textbf{Abstract.} For a set of primes $\pi$, denote by $E_\pi$ the class of finite groups containing a Hall $\pi$-subgroup. We establish that $E_{\pi_1}\cap E_{\pi_2}$ is contained in $E_{\pi_1\cap\pi_2}$. As a corollary, we prove that if $\pi$ is a set of primes, $l$ is an integer such that $2\leqslant l<|\pi|$ and $G$ is a finite group that contains a Hall $\rho$-subgroup for every subset $\rho$ of $\pi$ of size $l$, then $G$ contains a solvable Hall $\pi$-subgroup.




\section{Introduction}

Let $\pi$ be a set of primes. A $\pi$-subgroup $H$ of a finite group $G$ is called a Hall $\pi$-subgroup if the index $|G:H|$ is not divisible by elements of $\pi$. Denote by $E_\pi$ the condition that a finite group possesses a Hall $\pi$-subgroup. We will also denote by $E_\pi$ the class of groups satisfying this condition.

The well-known result of P. Hall \cite{Hall28} states that a solvable group satisfies $E_\pi$ for every set $\pi$ of primes. P. Hall \cite{Hall37} and independently S.A. Chunikhin \cite{Chun38} proved that a finite group is solvable whenever it has a $p$-complement (that is, a Hall $p'$-subgroup, where $p'$ is a complement to $\{p\}$ in the set of primes) for every prime $p$. Later P. Hall \cite{Hall56} conjectured that a finite group is solvable if it contains a Hall $\{p, q\}$-subgroup for every pair of primes $p$ and $q$. Z. Arad and M.B.~Ward proved this conjecture in \cite{AraWar82} using the classification of finite simple groups. Their result has several generalizations: V.N. Tyutyanov \cite{Tyu02} proved that a finite group is solvable if it has a Hall $\{2, p\}$-subgroup for every prime $p$; A.A. Buturlakin and A.P. Khramova \cite{ButKhr20} proved that a finite group $G$ has a solvable Hall $\pi$-subgroup if it has a Hall $\{p ,q\}$-subgroup for every pair $p$, $q\in\pi$. The latter generalization of the Arad-Ward theorem suggests the question of whether the corresponding generalization of the Hall--Chunikhin theorem holds, that is, whether  a finite group has a solvable Hall $\pi$-subgroup if it has a Hall $\pi\setminus\{p\}$-subgroup for every $p\in\pi$ and $|\pi|\geqslant 3$. Here we prove a more general statement.

\begin{theo}\label{Generalization} Let $\pi$ be a set of primes of size $k\geqslant 3$ and let $l$ be an integer such that $2\leqslant l\leqslant k$. Assume that a finite group $G$ has a Hall $\rho$-subgroup for every $\rho\subseteq\pi$ with $|\rho|=l$. Then $G$ has a solvable Hall $\pi$-subgroup.
\end{theo}

This theorem is a corollary of the main result of \cite{ButKhr20} and the following statement which gives an affirmative answer to Question 18.38 of \cite{Kou} by A.V. Zavarnitsine.

\begin{theo}\label{t:main} Let $\pi_1$ and $\pi_2$ be sets of primes. Then $E_{\pi_1}\cap E_{\pi_2}\subseteq E_{\pi_1\cap\pi_2}$.
\end{theo}

Indeed, it is clear that if $G$ contains a Hall $\rho$-subgroup for every subset $\rho$ of $\pi$ of size $l\geqslant 2$, then it contains a Hall $\{p, q\}$-subgroup for all $p$, $q\in\pi$ by Theorem~\ref{t:main}. Now the main result of \cite{ButKhr20} implies that $G$ has a solvable Hall $\pi$-subgroup.

Observe that the symmetric group $\operatorname{Sym}_7$ contains a unique conjugacy class of Hall $\{2, 3, 5\}$-subgroup consisting of point stabilizers and a unique conjugacy class of Hall $\{2, 3\}$-subgroups consisting of subgroups isomorphic to $\operatorname{Sym}_4\times\operatorname{Sym}_4$, but $\operatorname{Sym}_6$ does not contain a Hall $\{2, 3\}$-subgroup. Hence a Hall $\{2, 3\}$-subgroup of $\operatorname{Sym}_7$ cannot be an intersection of Hall $\{2, 3\}$- and $\{2, 3, 5\}$-subgroups. Thus one cannot state that a Hall $\pi_1\cap\pi_2$-subgroup in Theorem~\ref{t:main} appears as an intersection of Hall $\pi_1$- and $\pi_2$-subgroups.

It is worth noting that Theorem~\ref{t:main} has another interesting corollary. For a finite group $G$, denote by $\pi(G)$ the set of prime divisors of the order of $G$ and denote by $\Pi(G)$ the set of all subsets $\pi$ of $\pi(G)$ such that $G\in E_\pi$. In this notation, Theorem 2 is equivalent to saying that $\Pi(G)$ is a lower subsemilattice of the subset lattice of $\pi(G)$. It is easy to see that a finite lower semilattice with the greatest element is a lattice. So Theorem~\ref{t:main} and  the fact that $\pi(G)$ is the greatest element in $\Pi(G)$ imply the following corollary.

\begin{cor}\label{c:Lattice} The set $\Pi(G)$ forms a lattice under the set inclusion.
\end{cor}

Observe that the lattice meet of two elements in $\Pi(G)$ is the ordinary meet of two subsets, but the lattice join is not (for example, the Arad--Ward theorem implies that for every non-solvable group $G$, there exist two primes $p$ and $q$ such that $G$ does not have a Hall $\{p, q\}$-subgroup). It seems to be difficult to describe the latter operation in general situation.

\section{Preliminaries}

Let $\pi$ be a set of prime numbers. Denote by $\pi'$ the complement of $\pi$ in the set of all primes. For an integer $n$, denote by $n_\pi$ the $\pi$-part of $n$, that is, the greatest divisor of $n$ all of whose prime divisors lie in $\pi$.

For integers $n_1$,$\dots$, $n_s$ denote by $(n_1,\dots, n_s)$ their greatest common divisor. For a real number $x$, denote by $\lfloor x\rfloor$ the integral part of $x$, that is, the largest integer less than or equal to $x$. For an odd integer $q$, define $\varepsilon(q)=(-1)^{\frac{q-1}{2}}$. We denote by $\Phi_n(x)$ the $n$th cyclotomic polynomial

For a subgroup $K$ of a general linear group $GL_n(q)$, denote by $PK$ the projective image of $K$, that is, the image of $K$ in $PGL_n(q)$. Put $GL_n^+(q)=GL_n(q)$ and $GL_n^-(q)=GU_n(q)$.

\begin{lemma}\label{l:GL2}\cite[Corollary 8.11]{VdRev11} Suppose that $G=GL^\eta_2(q)$, where $\eta\in\{+, -\}$ and $q$ is a power of a prime $p$. Set $\varepsilon=\varepsilon(q)$. Let $\pi$ be a set
of primes such that $2, 3\in\pi$ and $p\not\in\pi$. A subgroup $H$ of $G$ is a Hall $\pi$-subgroup if and only if one of the following holds:
\begin{itemize}
\item[$(1)$] $\pi\cap\pi(G)\subseteq\pi(q-\varepsilon)$ and $PH$ is a Hall $\pi$-subgroup of a dihedral subgroup
$D_{2(q-\varepsilon)}$ of order $2(q-\varepsilon)$ of the group $PG$;
\item[$(2)$] $\pi\cap\pi(G) = \{2, 3\}$, $(q^2-1)_{\{2,3\}} = 24$, and $PH\simeq\operatorname{Sym}_4$.
\end{itemize}
Futhermore, two Hall $\pi$-subgroups of $G$ satisfying the same condition $(1)$ or $(2)$ are conjugate.
\end{lemma}

Note that Lemma~\ref{l:GL2} implies, in particular, that all Hall subgroups of $GL^\eta_2(q)$ are solvable.

\begin{lemma}\label{l:SL2}\cite[Lemma 8.10]{VdRev11} Suppose that $G=SL_2(q)$, where $q$ is a power of a prime $p$. Let $\pi$ be a set
of primes such that $2, 3\in\pi$ and $p\not\in\pi$. Then $G$ contains a Hall $\pi$-subgroup $H$ if and only if one of the following holds:
\begin{itemize}
\item[$(1)$] $\pi\cap\pi(G)\subseteq\pi(q-\varepsilon(q))$ and $H$ is a Hall $\pi$-subgroup of a dihedral group $D_{(q-\varepsilon(q))}$;
\item[$(2)$] $\pi\cap\pi(G)=\{2, 3\}$, $(q^2-1)_{\{2, 3\}}=24$ or $48$, and $H\simeq \operatorname{Alt}_4$ or $\operatorname{Sym}_4$;
\item[$(3)$] $\pi\cap\pi(G)=\{2, 3, 5\}$, $(q^2-1)_{\{2, 3, 5\}}=120$, and $H\simeq\operatorname{Alt}_5$.
\end{itemize}

Hall $\pi$-subgroups appearing in Item $(1)$ form a single conjugacy class in $G$.
\end{lemma}

\begin{lemma}\label{l:Sym}\cite[Theorem 8.1]{VdRev11} Let $H$ be a proper Hall $\pi$-subgroup of $\operatorname{Sym}_n$. Then one of the following holds.
\begin{itemize}
\item[$(1)$] $n$ is a prime and $\pi\cap\pi(\operatorname{Sym}_n)=\pi((n-1!))$;
\item[$(2)$] $n=7$, or $8$, and $\pi\cap\pi(\operatorname{Sym}_n)=\{ 2, 3\}$.
\end{itemize}
\end{lemma}

\begin{lemma}\label{l:3notInPi}\cite[Lemma 10]{Guo15} Suppose that $S$ is a finite simple group, $\pi$ is a set of primes such that $\pi(S)\not\subseteq\pi$ and $3\not\in\pi$. If $H$ is a Hall $\pi$-subgroup of $S$, then $H$ is solvable.
\end{lemma}

Denote by $E^{ns}_\pi$ the subclass of $E_\pi$ consisting of groups containing a non-solvable Hall $\pi$\nobreakdash-subgroup.


\begin{prop}\label{p:ListOfNonsolvableHallSubgroups} Let $\pi$ be a set of primes and let $S$ be a finite nonabelian simple group. Then $S\in E^{ns}_\pi$ if and only if either $S$ is a $\pi$-group or the pair $(S, \pi)$ satisfies one of the conditions of Table~1. Furthermore, if $(S,\pi)$ satisfies the conditions of the row marked by $+$, then $S$ contains a conjugacy class of Hall $\pi$-subgroups which remains to be a conjugacy class in $\operatorname{Aut}(S)$.
\end{prop}

\begin{center}

{\small\begin{longtable}{|c|c|c|c|}
\caption{Non-solvable proper Hall subgroups of finite simple groups}\\ \hline
$S$&Conditions&$\pi\cap\pi(S)$& \\ \hline
\endfirsthead

\multicolumn{4}{c}%
{{\tablename\ \thetable{} -- continued from previous page}} \\
\hline $S$&Conditions&$\pi\cap\pi(S)$&\\ \hline
\endhead

\hline \multicolumn{4}{|r|}{{Continued on next page}} \\ \hline
\endfoot

\endlastfoot
$\operatorname{Alt}_n$& $n\geqslant7$ is prime& $\pi((n-1)!)$&\\ \hline
$A_{n-1}(q)$& $n$ is an odd prime, $(n, q-1)=1$& $\pi(S)\setminus\pi\left(\Phi_n(q)\right)$&\\ \cline{2-4}
$q=p^\alpha$& $n=4$, $(6, q-1)=1$, $q>2$& $\pi(S)\setminus\pi\left(\Phi_3(q)\Phi_4(q)\right)$&\\ \cline{2-4}
& $n=5$, $(10, q-1)=1$& $\pi(S)\setminus\pi\left(\Phi_4(q)\Phi_5(q)\right)$&\\ \cline{2-4}
& $n=5$, $(30, q-1)=1$, $q>2$& $\pi(S)\setminus\pi\left(\Phi_3(q)\Phi_4(q)\Phi_5(q)\right)$&\\ \cline{2-4}
& $n=7$, $(35, q-1)=1$, $(3, q+1)=1$& $\pi(S)\setminus\pi\left(\Phi_5(q)\Phi_6(q)\Phi_7(q)\right)$&\\ \cline{2-4}
& $n=8$, $(70, q-1)=1$, $(3, q+1)=1$& $\pi(S)\setminus\pi\left(\Phi_4(q)\Phi_5(q)\Phi_6(q)\Phi_7(q)\right)$&\\ \cline{2-4}
& $n=11$, $(462, q-1)=1$, $(5, q+1)=1$& $\pi(S)\setminus\pi\left(\Phi_7(q)\Phi_8(q)\Phi_9(q)\Phi_{10}(q)\Phi_{11}(q)\right)$&\\ \cline{2-4}
& $n=2$, $(q^2-1)_{\{2, 3, 5\}}=120$& $\{2, 3, 5\}$&\\ \cline{2-4}
& $p\not\in\pi$, $(12, q-1)=12$, $\operatorname{Sym}_n\in E^{ns}_\pi$, & \multirow{2}{*}{ $\{2, 3\}\subseteq\pi\cap\pi(S)\subseteq\pi(q-1)\cup\pi(n!)$}&\multirow{2}{*}{+}\\
& if $r\in(\pi\cap\pi(n!))\setminus\pi(q-1)$, then $|S|_r=(n!)_r$& &\\ \cline{2-4}
& $(3, q+1)=3$, $GL_2(q)\in E_\pi$, & \multirow{2}{*}{$\{2, 3\}\subseteq\pi\cap\pi(S)\subseteq\pi(q^2-1)$}&\\
& $\operatorname{Sym}_{\lfloor\frac{n}{2}\rfloor}\in E^{ns}_\pi$& &\\\cline{2-4}
& $n=4$, $(8, q-5)=8$, $(q+1)_3=3$, $(q^2+1)_5=5$& $\{2, 3, 5\}$&\\ \hline
${}^2A_{n-1}(q)$& $p\not\in\pi$, $(12, q+1)=12$, $\operatorname{Sym}_n\in E^{ns}_\pi$ &\multirow{2}{*}{ $\{2, 3\}\subseteq\pi\cap\pi(S)\subseteq\pi(q+1)\cup\pi(n!)$}&\multirow{2}{*}{+}\\
$q=p^\alpha$& if $r\in(\pi\cap\pi(n!))\setminus\pi(q+1)$, then $|S|_r=(n!)_r$& &\\ \cline{2-4}
& $(3, q-1)=3$, $GU_2(q)\in E_\pi$, & \multirow{2}{*}{$\{2, 3\}\subseteq\pi\cap\pi(S)\subseteq\pi(q^2-1)$}&\\
& $\operatorname{Sym}_{\lfloor\frac{n}{2}\rfloor}\in E^{ns}_\pi$&& \\\cline{2-4}
& $n=4$, $(8, q+5)=8$, $(q-1)_3=3$, $(q^2+1)_5=5$& $\{2, 3, 5\}$&\\ \hline
$B_{n}(q)$&$(12, q-\varepsilon(q))=12$, $\operatorname{Sym}_n\in E^{ns}_\pi$&$\{2, 3\}\subseteq\pi\cap\pi(S)\subseteq\pi(q-\varepsilon(q))$&+\\ \cline{2-4}
$q=p^\alpha$&$n=3$, $p\not\in\pi$, $|S|_\pi=2^9\cdot3^4\cdot5\cdot7$&$\{2, 3, 5, 7\}$&\\ \cline{2-4}
&$n=4$, $p\not\in\pi$, $|S|_\pi=2^{14}\cdot3^5\cdot5^2\cdot7$&$\{2, 3, 5, 7\}$&\\ \hline
$C_{n}(q)$&$SL_2(q)$, $\operatorname{Sym}_n\in E_\pi$ and $SL_2(q)\times \operatorname{Sym}_n\in E_{\pi}^{ns}$&$\{2, 3\}\subseteq\pi\cap\pi(S)\subseteq\pi(q^2-1)$&\\ \hline
$D_{n}(q)$&$p=2$, $n$ is a Fermat prime, &\multirow{2}{*}{$\pi(S)\setminus\pi\left(\frac{q^n-1}{q-1}(q^{n-1}+1)\right)$}&\\
$q=p^\alpha$&$(n, q-1)=1$&&\\ \cline{2-4}
&$\varepsilon(q)^n=1$, $(12, q-\varepsilon(q))=12$, &\multirow{2}{*}{$\{2, 3\}\subseteq\pi\cap\pi(S)\subseteq\pi(q-\varepsilon(q))$}&\multirow{2}{*}{+}\\
&$\operatorname{Sym}_n\in E^{ns}_\pi$ &&\\ \cline{2-4}
&$\varepsilon(q)^n=-1$, $(12, q-\varepsilon(q))=12$,  &\multirow{2}{*}{$\{2, 3\}\subseteq\pi\cap\pi(S)\subseteq\pi(q-\varepsilon(q))$}&\multirow{2}{*}{+}\\
&$\operatorname{Sym}_{n-1}\in E^{ns}_\pi$ &&\\ \cline{2-4}
&$n=4$, $p\not\in\pi$, $|S|_\pi=2^{13}\cdot3^5\cdot5^2\cdot7$&$\{2, 3, 5, 7\}$&\\ \hline
${}^2D_{n}(q)$&$p=2$, $n-1$ is a Mersenne prime, &\multirow{2}{*}{$\pi(S)\setminus\pi\left(\frac{q^{n-1}-1}{q-1}(q^{n}+1)\right)$}&\\
$q=p^\alpha$&$(n-1, q-1)=1$&&\\ \cline{2-4}
&$\varepsilon(q)^n=-1$, $(12, q-\varepsilon(q))=12$,  &\multirow{2}{*}{$\{2, 3\}\subseteq\pi\cap\pi(S)\subseteq\pi(q-\varepsilon(q))$}&\multirow{2}{*}{+}\\
&$\operatorname{Sym}_n\in E^{ns}_\pi$&&\\ \cline{2-4}
&$\varepsilon(q)^n=1$, $(12, q-\varepsilon(q))=12$, &\multirow{2}{*}{$\{2, 3\}\subseteq\pi\cap\pi(S)\subseteq\pi(q-\varepsilon(q))$}&\multirow{2}{*}{+}\\
&$\operatorname{Sym}_{n-1}\in E^{ns}_\pi$ &&\\ \hline
$G_{2}(q)$& $(q^2-1)_{\{2, 3, 7\}}=24$, $(q^4+q^2+1)_7=7$&$\{2, 3, 7\}$&+\\ \hline
$E_{6}(q)$& &$\{2, 3, 5\}\subseteq\pi\cap\pi(S)\subseteq\pi(q-1)$&+\\ \hline
${}^2E_{6}(q)$& &$\{2, 3, 5\}\subseteq\pi\cap\pi(S)\subseteq\pi(q+1)$&+\\ \hline
$E_{7}(q)$& &$\{2, 3, 5, 7\}\subseteq\pi\cap\pi(S)\subseteq\pi(q-\varepsilon(q))$&+\\ \hline
$E_{8}(q)$& &$\{2, 3, 5, 7\}\subseteq\pi\cap\pi(S)\subseteq\pi(q-\varepsilon(q))$&+\\ \hline
$M_{11}$& &$\{2, 3, 5\}$&\\ \hline
$M_{22}$& &$\{2, 3, 5\}$&\\ \hline
$M_{23}$& &$\{2, 3, 5\}$, $\{2, 3, 5, 7\}$, $\{2, 3, 5, 7, 11\}$&\\ \hline
$M_{24}$& &$\{2, 3, 5\}$&\\ \hline
$J_{1}$& &$\{2, 3, 5\}$&\\ \hline
$J_{4}$& &$\{2, 3, 5\}$&\\ \hline

\end{longtable}}
\end{center}
\begin{proof} The contents of Table~1 is mostly an extraction from the classification of Hall subgroups of finite simple groups. This classification is obtained in a series of papers by a number of authors and presented in a convenient form in \cite[Appendix 1]{VdRev11} together with the corresponding references.

Let us make some remarks on how results of that paper are applied here. By the Feit--Thompson theorem, we may assume that $2\in\pi$. Lemma~\ref{l:3notInPi} implies that $3$ is also an element of $\pi$. If $S$ is an alternating or sporadic groups, then the proposition follows from \cite[Theorems 8.1 and 8.2]{VdRev11}. If $S$ is a group of Lie type over a field of characteristic $p$ and $p\in\pi$, then \cite[Theorem~8.3]{VdRev11} states that a Hall $\pi$-subgroup is contained in a Borel subgroup, which is solvable, or is a parabolic subgroup of $S$. In the latter case $S$ is either linear, or orthogonal of even dimension, and we apply \cite[Theorems 8.6 or 8.5]{VdRev11} in these cases respectively. Finally, if $p\in\pi$, then the corresponding classification is given in \cite[Theorems 8.12-15]{VdRev11}.
\end{proof}

Note that the absence of $+$ in a row of Table~1 do not imply any statements on Hall $\pi$-subgroups of $S$, that is, $S$ still can have a conjugacy class of Hall $\pi$-subgroups which is invariant under the automorphisms of~$S$.

\begin{cor}\label{c:SameRow} Let $S$ be a finite simple group. Suppose that $\pi_1$ and $\pi_2$ are two sets of primes such that $(S, \pi_1)$ and $(S, \pi_2)$ satisfy the restrictions of the same row of Table 1. Then the pair $(S, \pi_1\cap\pi_2)$ also satisfies these restrictions, and, in particular, $S\in E_{\pi_1\cap\pi_2}$.
\end{cor}

\begin{proof} It is easily seen that $(\pi_1\cap\pi_2)\cap\pi(S)$ satisfies the condition of the third column of Table~1 whenever $\pi_1\cap \pi(S)$ and $\pi_2\cap\pi(S)$ do. Every condition of the second column regarding $\pi$ is either a universal property of elements of $\pi$, or of the form ``$H\in E_\pi  (E_\pi^{ns})$'', where $H$ is one of the groups  $\operatorname{Sym}_k$, $GL_2^\pm(q)$, $SL_2(q)$. So the desired result follows from Lemmas~\ref{l:GL2}, \ref{l:SL2}, and \ref{l:Sym}.
\end{proof}

Recall that for a group $G$ and its subgroups $K$ and $H$ such that $K$ is a normal subgroup of $H$, the group of induced automorphism $\operatorname{Aut}_G(H/K)$ of the section $H/K$ is the image of $N_G(H)\cap N_G(K)$ in the automorphism group $\operatorname{Aut}(H/K)$.

\begin{lemma}\label{l:EPiCriterion}\cite[Corollary 4.6]{VdRev11} Let $$1=G_0<G_1<\dots<G_n=G$$ be a composition series of a finite group $G$, which is a refinement of a chief series. Let $\pi$ be a set of primes. Then $G\in E_\pi$ if and only if $\operatorname{Aut}_G(G_i/G_{i-1})\in E_\pi$ for $i=1,\dots, n$.
\end{lemma}

\begin{lemma}\label{l:EPiAlmostSimple}\cite[Proposition 4.8]{VdRev11} Let $\pi$ be a set of primes and let $S$ be a normal $E_\pi$-subgroup of a finite group $G$ such that $G/S$ is a $\pi$-group. Suppose that $M$ is a Hall $\pi$-subgroup of $S$. Then the following statements are equivalent.
\begin{itemize}
\item[$(1)$] There exists a Hall $\pi$-subgroup $H$ of $G$ such that $M=H\cap S$.
\item[$(2)$] The conjugacy class $M^S$ is $G$-invariant.
\end{itemize}
\end{lemma}

\section{Proof of Theorem~\ref{t:main}}

Let $G$ be a minimal counterexample to Theorem~\ref{t:main}. Then $G$ is almost simple by Lemma~\ref{l:EPiCriterion}. Let $S$ be the minimal normal subgroup of $G$ and set $\tau_i=\pi_i\cap \pi(S)$ for $i=1, 2$.  Suppose that $\tau_1\subseteq\tau_2$. Then $\pi_1\cap\pi_2\cap\pi(S)=\tau_1$. Let $H$ be a Hall $\pi_1$-subgroup of $G$ and $L=HS$. Then $H\cap S$ is a Hall $\tau_1$-subgroup of $S$ and $(H\cap S)^S=(H\cap S)^L$ by Lemma~\ref{l:EPiAlmostSimple}. Let $\overline K$ be a Hall $\pi_1\cap\pi_2$-subgroup of $L/S$ and $K$ its preimage in $L$. Then $(H\cap S)^S=(H\cap S)^K$ and Lemma~\ref{l:EPiAlmostSimple} implies that $G\in E_{\pi_1\cap\pi_2}$. Thus we may assume that the intersection $\tau_1\cap\tau_2$ is distinct from both $\tau_1$ and $\tau_2$. In particular, $\tau_1$ and $\tau_2$ are both proper subsets of $\pi(S)$.

By Hall's theorem, it follows that all Hall $\pi_1$- and $\pi_2$-subgroups of $G$ are non-solvable (we will use only the fact that there are non-solvable Hall $\pi_1$- and $\pi_2$-subgroups, but not the absence of solvable ones). Then $S$ also contains non-solvable Hall $\pi_1$- and $\pi_2$-subgroups and either $S\not\in E_{\pi_1\cap\pi_2}$, or $S\in E_{\pi_1\cap\pi_2}$ but $G\not\in E_{\pi_1\cap\pi_2}$.

We consider nonabelian simple groups $S$ in turn. First, we extract sets $\tau$ such that $\tau\subset\pi(S)$ and $S\in E_\tau^{ns}$ from Table~1. Then we determine whether there are incomparable $\tau_1$ and  $\tau_2$ among them. If some incomparable $\tau_1$ and $\tau_2$ arise from the same line of Table~1 and this line has +, then $S$ has a conjugacy class of Hall ${\pi_1\cap\pi_2}$-subgroups which is also a conjugacy class in $\operatorname{Aut}(S)$. Hence $G\in E_{\pi_1\cap\pi_2}$ by Lemma~\ref{l:EPiAlmostSimple}. Thus we may assume that $\tau_1$ and $\tau_2$ arise either from the row not marked by +, or from the different rows.

If $S=Alt_n$, then $\tau=\pi((n-1)!)$, and there is nothing to prove. The only sporadic group with more than one set $\tau$ is $M_{23}$, in which case these sets $\tau$ form a nested chain. Hence $S$ is a~group of Lie type. Furthermore, by the last observation of the previous paragraph, we may assume that $S$ is classical.



\textbf{1.} Let $S=A_{n-1}(q)$, where $q$ is a power of a prime $p$.

If $n=2$, then $\tau=\{2, 3, 5\}$ since $\operatorname{Sym}_n$ and $\operatorname{Sym}_{\lfloor\frac{n}{2}\rfloor}$ are solvable.

Assume that $n=4$. Then again $\operatorname{Sym}_n$ and $\operatorname{Sym}_{\lfloor \frac{n}{2}\rfloor}$ are solvable, and hence one of the following holds:
\begin{itemize}
\item $(6, q-1)=1$, $q>2$  and $\tau=\pi(S)\setminus\pi\left((q^2+q+1)(q^2+1)\right)$;

\item  $(8, q-5)_2=8$, $(q+1)_3=3$, $(q^2+1)_5=5$, and $\tau=\{2, 3, 5\}$.
\end{itemize}

Since the conditions $(6, q-1)=1$ and $(8, q-5)_2=8$ are not compatible, there is at most one $\tau$ for $S$.

Assume that $n=5$. Then four different types of $\tau$ can appear:

\begin{itemize}
\item $(5, q-1)=1$ and $\tau=\pi(S)\setminus\pi\left(\frac{q^5-1}{q-1}\right)$;

\item $(10, q-1)=1$ and $\tau=\pi(S)\setminus\pi\left(\frac{(q^5-1)(q^4-1)}{(q-1)(q^2-1)}\right)$;

\item $(30, q-1)=1$, $q>1$ and $\tau=\pi(S)\setminus\pi\left(\frac{(q^5-1)(q^4-1)(q^3-1)}{(q-1)^2(q^2-1)}\right)$;

\item $(12, q-1)=12$, $p\not\in\tau$ and $\{2,3\}\subseteq\tau\subseteq\pi(q-1)\cup\{5\}$.
\end{itemize}

Observe that the sets $\tau$ of types 1--3 form a nested chain. Hence we may assume that $(12, q-1)=12$. Since the row corresponding to the last type is marked by +, we may assume that only one of $\tau_1$ and $\tau_2$, say $\tau_2$, is of the last type. Then $(5, q-1)=1$,  $\tau_1=\pi(S)\setminus \pi\left(\frac{q^5-1}{q-1}\right)$ and $\tau_2\subseteq\pi(q-1)\cup\{5\}$. However, $\pi(q-1)\cup\{5\}\subseteq\pi(S)\setminus\pi\left(\frac{q^5-1}{q-1}\right)$. Indeed, every prime of the left-hand set divides $p(q^4-1)$ which in turn divides the order of $S$ and is coprime with $\frac{q^5-1}{q-1}$. It follows that $\tau_2\subseteq\tau_1$, contrary to our assumption.

Assume that $n=7$. Then one of the following holds:
\begin{itemize}

\item $(7, q-1)=1$ and $\tau=\pi(S)\setminus\pi\left(\frac{q^7-1}{q-1}\right)$;

\item $(35, q-1)=(3, q+1)=1$ and $\tau=\pi(S)\setminus\pi\left(\frac{(q^3+1)(q^5-1)(q^7-1)}{(q-1)^2(q+1)}\right)$;

\item $(12, q-1)=12$, $p\not\in\tau$, $\operatorname{Sym}_7\in E_\tau^{ns}$ and $\{2, 3\}\subseteq\tau\subseteq\pi(q-1)\cup\{5, 7\}$.
\end{itemize}

Note that $\operatorname{Sym}_7\in E_\tau^{ns}$ if and only if $\tau\cap\{2,3,5,7\}$ is either
$\{2,3,5,7\}$ or $\{2,3,5\}$. Arguing as in the case $n=5$, we are left with the situation when $\tau_1$ and $\tau_2$ are sets of types 2 and 3, respectively. Since $\tau_2\subseteq\pi(q-1)\cup\{5, 7\}$, either $\tau_2\subseteq\tau_1$, or $7$ divides $q^3+1$, $7\in\tau_2$ and $\tau_2\setminus\{7\}\subseteq\tau_1$. In the latter case, the set $\tau_1\cap\tau_2=\tau_2\setminus\{7\}$ also satisfies conditions of type 3, and so $S\in C_{\pi_1\cap\pi_2}$ by Corollary~\ref{c:SameRow} and $G\in E_{\pi_1\cap\pi_2}$ by Lemma~\ref{l:EPiAlmostSimple}.

Assume that $n=8$. There are at most two types of non-solvable Hall subgroups: one requires $(70, q-1)=1$, and another one requires $(12, q-1)=12$. These conditions are incompatible. So either $\tau$ is uniquely determined, or $\tau_1$ and $\tau_2$ arise form the same row marked by +.

Assume that $n=11$. Then one of the following holds:

\begin{itemize}
\item $(11, q-1)=1$ and $\tau=\pi(S)\setminus\pi\left(\frac{q^{11}-1}{q-1}\right)$;

\item $(462, q-1)=(5, q+1)=1$ and  $\tau=\pi(S)\setminus\pi\left(\frac{(q^5+1)(q^7-1)(q^8-1)(q^9-1)(q^{11}-1)}{(q-1)^2(q^3-1)(q^4-1)(q+1)}\right)$;

\item $p\not\in\pi$, $(12, q-1)=12$, $\operatorname{Sym}_{11}\in E^{ns}_\tau$, and $\{2, 3\}\subseteq\tau\subseteq\pi(q-1)\cup\{5, 7, 11\}$;

\item $p\not\in\tau$, $(3, q+1)=3$, $GL_2(q)\in E_\tau$, $\operatorname{Sym}_{5}\in E^{ns}_\tau$ and $\{2, 3\}\subseteq\tau\subseteq\pi(q^2-1)$.
\end{itemize}

The set $\tau$ of type 1 contains all sets of other types, while
$\tau$ of type 2 contains $\tau$ of type 4. Also the conditions of type 4 are incompatible with the conditions of both types 2 and 3. Arguing as in the case $n=7$, one can handle the situation when $\tau_1$ and $\tau_2$ are of types 2 and 3, respectively. So we may assume that both $\tau_1$ and $\tau_2$ are of type 4. Corollary~\ref{c:SameRow} implies that $S$ contains a Hall $\tau_1\cap\tau_2$-subgroup. Furthermore, $\operatorname{Sym}_5\in E_{\tau_1\cap\tau_2}$,
and hence $\tau_1\cap\tau_2$ contains $2$, $3$, and $5$. By \cite[Lemma 2.5 and Lemma 4.3(c)]{RevVd11}, it follows that  Hall $\tau_1\cap\tau_2$-subgroups of $S$ are conjugate if and only if Hall $\tau_1\cap\tau_2$-subgroups of $GL_2(q)$ are conjugate. Since $5\in \tau_1\cap\tau_2$, the latter is true in view of Lemma~\ref{l:GL2}.

For the remaining values of $n$, there are at most three types of $\tau$. Namely, one of the following holds:

\begin{itemize}
\item $n$ is an odd prime, $(n, q-1)=1$, and $\tau=\pi(S)\setminus\pi\left(\frac{q^{n}-1}{q-1}\right)$;

\item $p\not\in\tau$, $(12, q-1)=12$, $\operatorname{Sym}_{n}\in E^{ns}_\tau$,  and $\{2, 3\}\subseteq\tau\subseteq\pi(q-1)\cup\pi(n!)$;

\item $p\not\in\tau$, $(3, q+1)=3$, $GL_2(q)\in E_\pi$, $\operatorname{Sym}_{\lfloor\frac{n}{2}\rfloor}\in E^{ns}_\tau$ and $\{2, 3\}\subseteq\tau\subseteq\pi(q^2-1)$.
\end{itemize}

The conditions on divisibility by $3$ imply that types 2 and 3 are incompatible. Also the sets $\tau$ of types 2 and 3 are contained in $\pi(S)\setminus\pi\left(\frac{q^{n}-1}{q-1}\right)$. So $\tau_1$ and $\tau_2$ are of type 4, and we repeat the argument used for $n=11$.

\textbf{2.} Let $S={}^2A_{n-1}(q)$. Observe that the conditions of the three corresponding rows of Table 1 are not compatible. So we are left with the case when
$(3,q-1)=3$ and both $\tau_1$, $\tau_2$ satisfy

\begin{itemize}
 \item $GU_2(q)\in E_{\tau}$, $\operatorname{Sym}_{\lfloor\frac{n}{2}\rfloor}\in E^{ns}_{\tau}$ and $\{2, 3\}\subseteq\tau\subseteq\pi(q^2-1)$.
\end{itemize}

In this case, the argument used for $A_{10}(q)$ works.

\textbf{3.} Let $S=C_n(q)$. Then every proper Hall subgroup of $S$ is contained in the projective image $PM$ of a subgroup $M$ isomorphic to the wreath product $\operatorname{SL}_2(q)\wr\operatorname{Sym}_n$ by Proposition~\ref{p:ListOfNonsolvableHallSubgroups} and \cite[Theorem 8.13]{VdRev11}. Moreover, Hall subgroups of $S$ are conjugate if and only if they are conjugate in $M$. By Proposition~\ref{p:ListOfNonsolvableHallSubgroups}, $\tau_1$ and $\tau_2$ are subsets of $\pi(q^2-1)$, which is a subset of $\pi(\operatorname{SL}_2(q))$. Since we assume that $S$ contains non-solvable Hall $\tau_1$- and $\tau_2$-subgroups, Lemmas~\ref{l:SL2} and~\ref{l:Sym} imply that each $\tau_i$ is either a subset of $\pi(q-\varepsilon(q))$, or is equal to $\{2, 3, 5\}$. Hence at least one of $\tau_1$ and $\tau_2$ is a subset of $\pi(q-\varepsilon(q))$, so is $\tau_1\cap\tau_2$. It follows from \cite[Lemma 8.10]{VdRev11} that $\operatorname{SL}_2(q)$ has one or three conjugacy classes of Hall $\tau_1\cap\tau_2$-subgroups. In both cases there is a conjugacy class which remains a conjugacy class in the automorphism group of $\operatorname{SL}_2(q)$. Now \cite[Theorem 8.13]{VdRev11} implies that $S$ contains a class of Hall $\tau_1\cap\tau_2$-subgroups which is also a conjugacy class in~$\operatorname{Aut}(S)$. Hence $G\in E_{\pi_1\cap\pi_2}$ by Lemma~\ref{l:EPiAlmostSimple}. Thus $S$ is not a symplectic group.

\textbf{4.} If $S$ is orthogonal, then the conditions of different rows of Table 1 are not compatible and the rows in which $\tau$ is not uniquely determined are marked by +.

Theorem \ref{t:main} is proved.

\section{Acknowledgments}

The authors were supported by Fundmental Research Funds for the Central Universities-President's Fund(JUSRP202406006). The research was carried out by the first author within the framework of the Sobolev Institute of Mathematics state contract (project FWNF-2022-0002).

\end{document}